\documentclass[12pt]{article}

\usepackage{graphicx} 
\usepackage{amssymb}
\usepackage{amsthm}
\usepackage{amsmath}
\usepackage{subfigure}

\usepackage{tikzit}

\tikzstyle{node}=[fill=black, draw=black, shape=circle, tikzit fill=black, tikzit draw=black, tikzit shape=circle, label={{[label distance=0.5cm]}}]
\tikzstyle{dots}=[fill=black, draw=black, shape=circle, inner sep=0.2ex, tikzit fill=black, tikzit draw=black]

\tikzstyle{edge}=[fill=none, draw=black, -, line width=2pt]

\newcommand{\degpoly}[2]{D(#1;#2)}
\newcommand{\be}{\begin{equation}}
\newcommand{\ee}{\end{equation}}
\newcommand{\bea}{\begin{eqnarray}}
\newcommand{\eea}{\end{eqnarray}}
\newcommand{\beas}{\begin{eqnarray*}}
\newcommand{\eeas}{\end{eqnarray*}}

\newtheorem{theorem}{Theorem}[section]

\newtheorem{proposition}[theorem]{Proposition}

\newtheorem{conjecture}[theorem]{Conjecture}

\title{On the Roots of Degree Polynomials}

\author{Jason I.\ Brown\\
jason.brown@dal.ca\\
Department of Mathematics and Statistics, Dalhousie University \\
Halifax, Nova Scotia, Canada. B3H 4R2\\
~\\
Ian C. George\\
i2george@uwaterloo.ca\\
Department of Combinatorics and Optimization, University of Waterloo\\
Waterloo, Ontario, Canada. N2L 3G1 }

\begin{document}

\date{}

\maketitle

\begin{center} \large Abstract \end{center}
The degree polynomial of a multigraph $G$ is given by $\sum _{v \in V(G)} x^{\mbox{deg}(v)}$. We investigate here properties of the roots of such polynomials. In addition to examining the roots for some families of graphs with few and many degrees, we provide some bounds on the moduli of the roots. We also propose a region that contains all roots for multigraphs of order $n$.

\vspace*{\baselineskip}

\vspace*{\baselineskip}

\section{Introduction}
\label{sec:intro}

Let $G = (V, E)$ be a multigraph (with possibly multiple edges but no loops) of order $n$ and size $m$, that is, $|V|=n$ and $|E| = m$ (if $G$ has no multiple edges, then we refer to it as a simple graph, or simply as a graph). For any $v \in V$, let $\mbox{deg}(v) = \mbox{deg}_G(v)$ denote its degree in $G$, and let $a_{k}$ be the number of vertices of degree $k$ in $G$.  The \emph{degree polynomial} of $G$ is defined by

\begin{eqnarray*} 
\degpoly{G}{x} & = & \sum_{v \in V} x^{\mbox{deg}(v)}\\
 & = & \sum_{k=\delta}^{\Delta} a_k x^k, 
\end{eqnarray*}
where $\Delta = \Delta(G)$ and $\delta = \delta(G)$ denote the largest and smallest degrees of $G$, respectively.  The degree polynomial encodes precisely the degree sequence of a graph, and hence contains the same information.  The polynomial (unlike many other graph polynomials) can clearly be constructed in linear time. 
This graph polynomial has been previously defined independently by a number of researchers \cite{canoy2014,jafarpour2020}.  Existing literature explores how it behaves under graph operations \cite{canoy2014}, and studies the polynomials for \emph{prime graphs} and their derivatives \cite{palahang2019polynomial,palahang2020differentiability}. 

The roots (or zeros) of many graph polynomials have been well studied\,---\,see, for example, roots of chromatic polynomials \cite{dong2005chromatic}, reliability polynomials \cite{brown1992}, and independence polynomials \cite{brown2000well,brown2001bounding,brown2004location}.  A natural question that arises is\,---\,why study roots of graph polynomials? On one hand, there is (and has been) considerable interest in roots of graph polynomials in their own right. Often the calculating the graph polynomial is intractable (as is the case for chromatic \cite{dong2005chromatic}, reliability \cite{BrownC92,RoyleSokal2004} and independence polynomials \cite{BrownHN03,ChudnovskySeymour2007}) as it encodes some information about the graph that is NP-hard (or NP-complete), and in this regard investigating the roots can lead to some non-trivial insight. However, even when determining the polynomial can be carried out efficiently, such as in our case (but others as well, such as Wiener polynomials \cite{wiener}), the location and the nature of the roots show much interesting structure. 

Moreover, mathematical structures as esoteric as fractals have arisen in a variety of contexts \cite{BrownHN03, browndegagnetwoterm}. The location and nature of roots of (graph) polynomials also have import on the shape of associated coefficient sequences. For example, an old result of Newton (cf.~\cite{Stanley1989}) reveals that if the roots of a polynomial $\displaystyle{f = \sum_{i=0}^{n} a_i x^i}$ with real coefficients has all real roots, then the coefficient sequence is {\em log-concave} ($a_i^2 \geq a_{-1}a_{i+1}$ for all $i$) and hence {\em unimodal} (i.e. $a_1 \leq a_2 \leq \cdots \leq a_k \geq a_{k+1} \geq  \cdots \geq a_n$ for some $k$).  Brenti et al. \cite{BrentiRoyleWagner1994} have extended this to polynomials with roots in the sector $\{z\in\mathbb{C}\colon\ 2\pi/3<\arg(z)<4\pi/3\}$ of the complex plane.  Both these results have been utilized numerous times, including Chudnovsky and Seymour's beautiful proof of the log concavity of the independence numbers of claw-free graphs \cite{ChudnovskySeymour2007}. As well, the roots of a positive polynomial lying outside a disk in the complex plane can imply that the coefficient sequence has a binomial shape. Michelen and Sahasrabudhe \cite{michelen} proved that for certain families of polynomials with non-negative coefficients, the absence of roots in a disc with center $z = 1$ implies that the coefficients are asymptotically normal, and this result has been applied \cite{Ralaivaosaona} to the generating function for the number of subtrees of a tree.

For all of these reasons, our attention has been drawn to the roots of degree polynomials, which we refer to as {\em degree roots}.  Degree roots for regular graphs are trivial: an $r$-regular graph with $n$ vertices has degree polynomial $nx^r$, and thus all its degree roots are $0$.  Slightly less trivial are the degree roots for graphs with two degrees: a graph with $a_{\Delta}$ vertices of degree $\Delta$ and $a_{\delta}$ vertices of degree $\delta$ has degree polynomial $a_{\Delta}x^{\Delta} + a_{\delta}x^{\delta}$.  This polynomial has $\delta$ roots at $0$, and the remaining $\Delta-\delta$ roots are the $(\Delta-\delta)^{th}$ roots of $-a_{\delta}/a_{\Delta}$.  Figure \ref{fig:deg_roots_2to10} shows degree roots for graphs of small order.  

\begin{figure}
     \centering
     \begin{subfigure}
         \centering
         \includegraphics[width=0.31\textwidth]{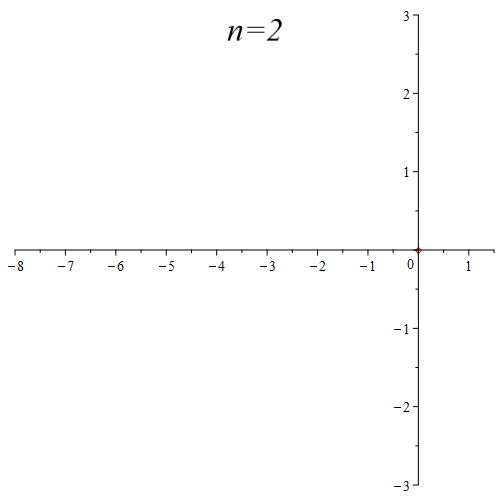}
     \end{subfigure}
     \begin{subfigure}
         \centering
         \includegraphics[width=0.31\textwidth]{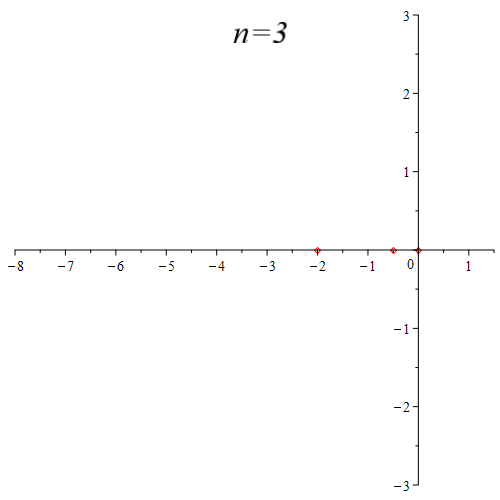}
     \end{subfigure}
     \begin{subfigure}
         \centering
         \includegraphics[width=0.31\textwidth]{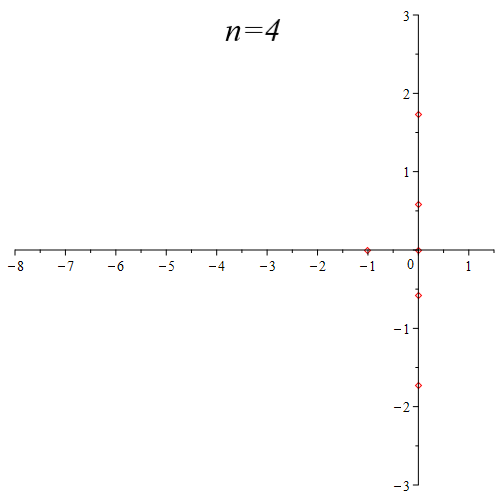}
     \end{subfigure}
     \begin{subfigure}
         \centering
         \includegraphics[width=0.31\textwidth]{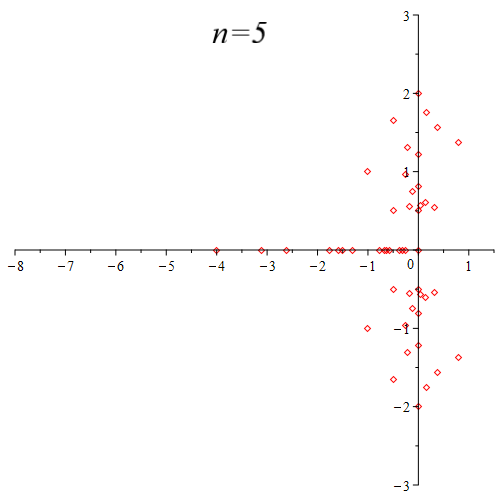}
     \end{subfigure}
     \begin{subfigure}
         \centering
         \includegraphics[width=0.31\textwidth]{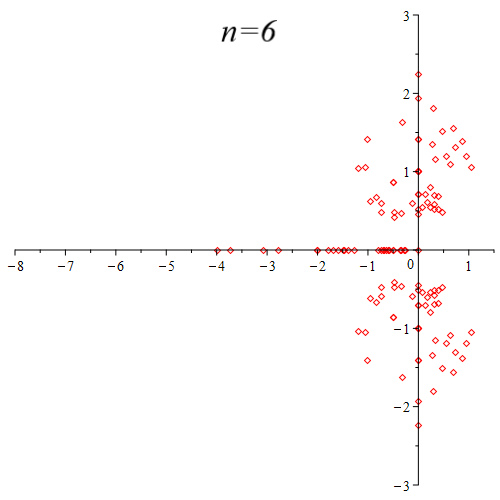}
     \end{subfigure}
     \begin{subfigure}
         \centering
         \includegraphics[width=0.31\textwidth]{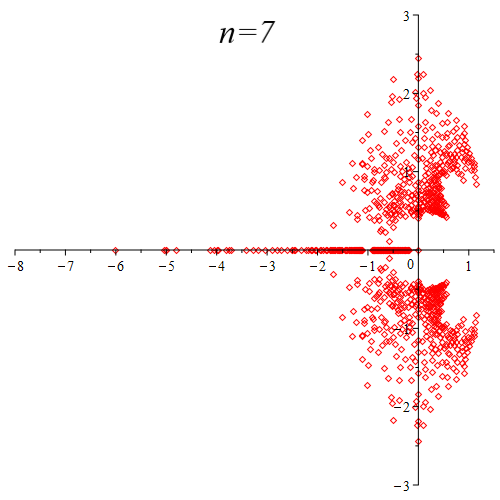}
     \end{subfigure}
     \begin{subfigure}
         \centering
         \includegraphics[width=0.31\textwidth]{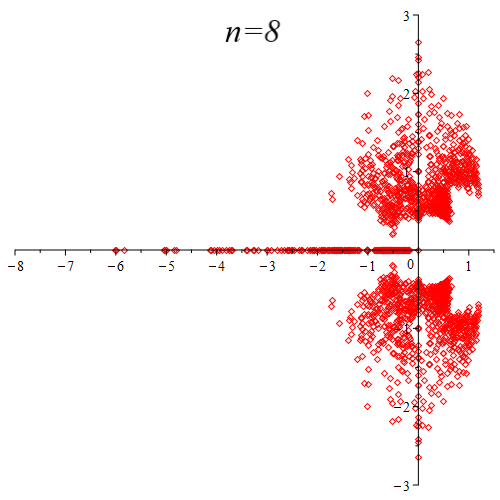}
     \end{subfigure}
     \begin{subfigure}
         \centering
         \includegraphics[width=0.31\textwidth]{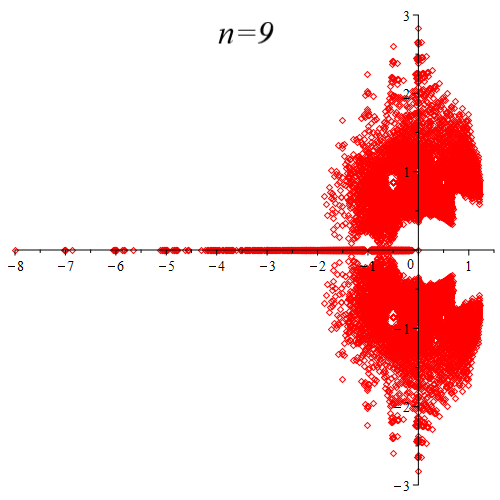}
     \end{subfigure}
     \begin{subfigure}
         \centering
         \includegraphics[width=0.31\textwidth]{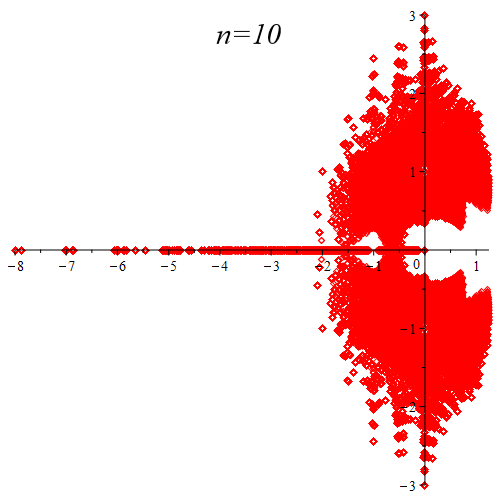}
     \end{subfigure}
        \caption{Degree roots of graphs of small order $n$.  }
        \label{fig:deg_roots_2to10}
\end{figure}

One important fact to note is that if graph $G$ of order $n$ has graph complement $\overline{G}$, then  $\degpoly{\overline{G}}{x} = x^{n-1} \cdot \degpoly{G}{1/x}$, so the set of nonzero degree roots of graphs of order $n$ is closed under inverses. In fact, the result holds for multigraphs  $M$ of order $n$ as well: Suppose that the maximum multiplicity of an edge is $\mu$. Then we can define the multigraph complement $\overline{M}$ as follows. For every pair of distinct vertices $x$ and $y$ with $k_{xy}$ edges between them in $M$, place a bundle of $\mu -k_{xy}$ parallel edges between the vertices in $\overline{M}$. Then it is easy to verify that $\degpoly{\overline{M}}{x} = x^{(n-1)\mu} \cdot \degpoly{M}{1/x}$.

Before proceeding, we observe that degree polynomials are polynomials with non-negative integer coefficients, that is, polynomials belonging to $\mathbb{Z}_{\geq 0}[x]$. Over the class of all multigraphs, the degree roots are not only contained in the set of all roots of $\mathbb{Z}_{\geq 0}[x]$, but equal to it, by the following reasoning. By a result of Hakimi \cite{hakimi1962realizability}, a sequence of non-negative integers $d_1 \geq \cdots \geq d_n$ is the degree sequence of a multigraph if and only if (i) $d_1 + \cdots + d_n$ is even, and (ii) $d_1 \leq \sum_{i=2}^n d_i$. It follows that in our terminology, a polynomial $p(x) \in \mathbb{Z}_{\geq 0}[x]$ is the degree polynomial of a multigraph if and only if (i) $p'(1)$ is even,  and (ii) $\mbox{deg}(p(x)) \leq p'(1)/2.$ It follows easily that if $f(x) \in \mathbb{Z}_{\geq 0}[x]$, then $2f(x)$ is the degree polynomial of a multigraph, and our result follows. The situation is not different even if we try to restrict to simple graphs, because if $M$ is a multigraph, then there exists a (simple) graph $G$ for which $\degpoly{M}{x} = K \cdot \degpoly{G}{x}$ for some constant $K$, and hence $G$ has precisely the same roots, including multiplicities, as $M$. (If $M$ has an edge $e$ with multiplicity $\mu_{e} > 1$, take $\mu_{e}$ disjoint copies of $M$ and let $e_1,\ldots, e_{\mu_{e}}$ be the copies of $e$ in each copy of $M$. Furthermore, let $u_i, v_i$ be the endpoints of $e_i$, $1 \leq i \leq \mu_{e}$. Consider the subgraph induced by the edges $\{e_1, \ldots, e_{\mu_{e}}\}$, where each vertex has an induced degree of $\mu_{e}$. Delete the edges in this induced subgraph and add new edges to create a $\mu_{e}$-regular induced subgraph that is simple. It is possible to do this by a folklore result since this induced subgraph has an even number of vertices and $\mu_{e} + 1 \leq 2 \mu_{e}$. If the graph is now simple, we are done. Otherwise, repeat this process for each edge of multiplicity greater than one until the graph is simple.)

A consequence of the set of degree roots being identical to the roots of polynomials with non-negative integer coefficients is that the closure of degree roots is the entire complex plane: Consider the polynomials 
$q^ax^{a+1}+p^ax \in \mathbb{Z}_{\geq 0}[x]$, where $p$ and $q$ are relatively prime positive integers and $a$ is a positive integer. It is easy to show that the set of non-trivial roots of such polynomials,
\[ \mathcal{A} = \left\{\omega_a \left(\frac{p}{q}\right) \;:\; a, p, q \in \mathbb{Z}_{\geq 1}, \; \text{gcd}(p, q)=1, \; (\omega_a)^a = -1 \right\}, \]
is dense in the complex plane, along the negative real axis, and along the imaginary axes (moreover, all such roots indeed arise as degree roots of simple graphs\,---\,see \cite{georgethesis}).

Yet this is not the end of the story, but just the beginning. What if we impose some graph-theoretic restrictions, such as restricting to certain families of multigraphs or graphs? Or restricting to graphs of fixed order, fixed size or fixed maximum degree -- that is, roots of degree polynomials with fixed sum of the coefficients, fixed evaluation of its derivative at $x = 1$, or fixed degree? For example, if we restrict to multigraphs of order $2$, then only $0$ is a degree root, while the roots of the corresponding polynomials $a +  bx^k \in \mathbb{Z}_{\geq 0}[x]$ where the sum of the coefficients $a+b = 2$ include all $n^{th}$ roots of $-1$ as well.
The results presented have been motivated by a few interesting problems, some of which arose out of similar investigations for other established graph polynomials, and others from the very nature of degree sequences.
\begin{itemize}
\item First, while the real roots of degree polynomials clearly must be negative, are they unbounded? How might they grow as a function of the order of the graph? 
\item The degree polynomials with the greatest number of nonzero terms can be characterized, as they arise from a well-characterized family of graphs. What can be said about the roots of such polynomials?
\item For graphs of order $n$, what regions bound the roots?
\end{itemize}
In terms of the latter, we conjecture a region that neatly contains the roots, whose bounding curve is unlike the usual circular disks that are presented to bound  the roots of other graph polynomials. 

We devote the next two sections to precisely such questions.

\section{Degree Roots for Graphs With Few or Many Degrees}
\label{sec:families}

In this section, we investigate the degree roots of two families of graphs, one with few degrees, the other with many.


\subsection{Complete Graphs with a Leaf}
\label{subsec:complete_with_leaf}

Of course, the degree roots of any regular multigraph are only $0$. The degree roots of multigraphs with only two degrees are almost just as trivial, being $0$ and $k^{th}$ roots of some negative integer, for some $k$. However, once we have more than two different degrees, the nature and location of the degree roots become more interesting. There is a broad range of graphs with exactly three different degree values, so we have chosen to focus on the degree roots of $CL_n$, a complete graph of order $n-1$ with a leaf attached to one vertex; clearly $\degpoly{CL_n}{x} = x^{n-1}+(n-2)x^{n-2}+x$. Figure \ref{fig:leaf_complete_all_roots} shows all roots for $\degpoly{CL_n}{x}$, $2 \leq n \leq 20$.  We can observe two things: there are roots which appear to be located near the negative integers, and there are roots which have modulus close to $1$.  Figure \ref{fig:leaf_complete_circ_roots} focuses on those roots that are within the unit circle.  It appears that the roots are approaching the entire unit circle from the inside, save for the roots at the origin and $-1$.  

\begin{figure}
    \centering
    \includegraphics[scale=0.3]{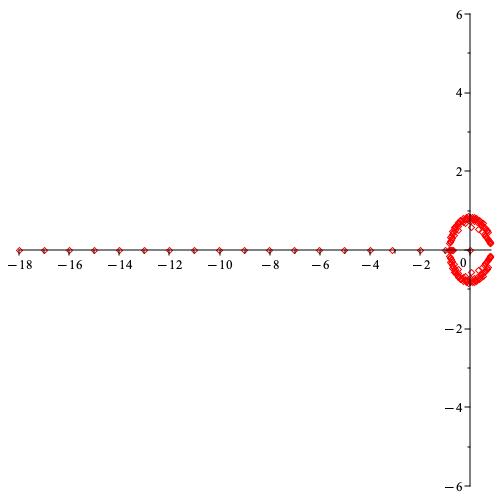}
    \caption{All roots of $\degpoly{CL_n}{x}$ for $2 \leq n \leq 20$.}
    \label{fig:leaf_complete_all_roots}
\end{figure}

\begin{figure}
    \centering
    \includegraphics[scale=0.3]{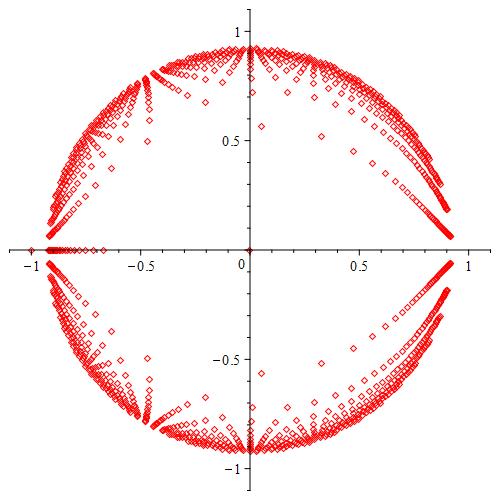}
    \caption{The roots of $\degpoly{CL_n}{x}$ for $2 \leq n \leq 50$ that are contained in the unit circle.  The roots appear to be converging outward to the unit circle as $n$ increases.}
    \label{fig:leaf_complete_circ_roots}
\end{figure}


Addressing the observation of the real roots, let us first count the negative real roots.  Consider the polynomial $\degpoly{CL_n}{-x}$:

\[ \degpoly{CL_n}{-x} = (-1)^{n-1}x^{n-1} + (-1)^{n-2}(n-2)x^{n-2} - x. \]

\noindent If $n$ is odd, the coefficients have exactly one sign change.  Thus $\degpoly{CL_n}{x}$ has exactly one negative root by Descartes' Rule of Signs (see, for example, \cite{anderson1998descartes}).  If $n$ is even, there are two sign changes in the coefficients.  Thus $\degpoly{CL_n}{x}$ has zero or two negative roots, also by the Rule of Signs.  We shall see there are in fact two negative roots for even $n$, except for $n=2$ when the degree polynomial is $\degpoly{CL_2}{x} = 2x$.  For $n \geq 4$, which is when $\degpoly{CL_n}{x}$ is a trinomial, we can locate a \emph{large} negative root within an error that vanishes as $n \to \infty$.

\begin{proposition}

Consider the graphs $CL_n$, $n \geq 4$.  For odd $n$, $\degpoly{CL_n}{x}$ has a real root in the interval $(-(n-2)-\epsilon_o(n), -(n-2))$ where

\[ \epsilon_o(n) = \frac{1}{(n-2)^{n-3}}. \]

\noindent For even $n$, $\degpoly{CL_n}{x}$ has a real root in the interval $(-(n-2), -(n-2)+\epsilon_e(n)]$, where

\[ \epsilon_e(n) = \frac{1}{(n-3)^{n-3}}. \]

\label{prop:leaf_complete_real}
\end{proposition}

\begin{proof}
To simplify some calculations, make the change of variables $x = (n-2)y$, and consider the polynomial

\[ 
    f(y) = \frac{1}{(n-2)^{n-1}}\degpoly{CL_n}{(n-2)y} = y^{n-1}+y^{n-2}+\frac{y}{(n-2)^{n-2}}.
\]


\noindent We first consider when $n$ is odd.  We will evaluate $f(y)$ at two points that give values with opposite sign, and apply the Intermediate Value Theorem (IVT).  The first point is $y = -1$:

\[ 
f(-1) = (-1)^{n-1} + (-1)^{n-2} + \frac{-1}{(n-2)^{n-2}}  = \frac{-1}{(n-2)^{n-2}}  < 0.
\]


\noindent The next point is $y = -1 - 1/(n-2)^{n-2}$:

\begin{align*}
    f\left(-1 - \frac{1}{(n-2)^{n-2}}\right) = & \; \left(-1 - \frac{1}{(n-2)^{n-2}}\right)^{n-1} + \left(-1 - \frac{1}{(n-2)^{n-2}}\right)^{n-2} \\
    & + \frac{-1 - \frac{1}{(n-2)^{n-2}}}{(n-2)^{n-2}} \\
    = & \; \frac{1 + \frac{1}{(n-2)^{n-2}}}{(n-2)^{n-2}}\left[\left(1 + \frac{1}{(n-2)^{n-2}}\right)^{n-3}-1\right] \\
    > &  \; 0
\end{align*}

\noindent since $(1 + 1/(n-2)^{n-2})^{n-3} > 1$.  Thus by the IVT, $f(y)$ has a root in the interval $(-1-1/(n-2)^{n-2}, -1)$.  Through the change of variables $x = (n-2)y$, $\degpoly{CL_n}{x}$ has a root in the interval $(-(n-2)-1/(n-2)^{n-3}, -(n-2)) = (-(n-2)-\epsilon_o(n), -(n-2))$.

Similarly, suppose that $n$ is even.  We still have that $f(-1) = -1/(n-2)^{n-2} < 0$.  Let us evaluate $f(y)$ at another point, namely $y = -1+1/(n-2)(n-3)^{n-3}$:

\begin{align*}
    f(y) = & \; f\left(-1+\frac{1}{(n-2)(n-3)^{n-3}}\right) \\
    = & \; \left(-1+\frac{1}{(n-2)(n-3)^{n-3}}\right)^{n-1} + \left(-1+\frac{1}{(n-2)(n-3)^{n-3}}\right)^{n-2} \\
    & + \frac{-1+\frac{1}{(n-2)(n-3)^{n-3}}}{(n-2)^{n-2}} \\
    = & \; \frac{\left(1-\frac{1}{(n-2)(n-3)^{n-3}}\right)^{n-2}}{(n-2)(n-3)^{n-3}} - \frac{1-\frac{1}{(n-2)(n-3)^{n-3}}}{(n-2)^{n-2}}.
\end{align*}

\noindent This quantity is non-negative, as

\begin{alignat*}{2}
    & & f\left(-1+\frac{1}{(n-2)(n-3)^{n-3}}\right) & \geq 0 \\
    & \iff & \frac{\left(1-\frac{1}{(n-2)(n-3)^{n-3}}\right)^{n-2}}{(n-2)(n-3)^{n-3}} & \geq \frac{1-\frac{1}{(n-2)(n-3)^{n-3}}}{(n-2)^{n-2}} \\
    & \iff & 1 - \frac{1}{(n-2)(n-3)^{n-3}} & \geq \frac{n-3}{n-2} \\
    & \iff & (n-2)(n-3)^{n-3}-1 & \geq (n-3)^{n-2} \\
    & \iff & 1 & \leq (n-3)^{n-3},
\end{alignat*}

\noindent and this last inequality is indeed true since $n \geq 4$.  Furthermore, there is equality if and only if $n=4$.  Applying the IVT, we conclude that $f(y)$ has a root in the interval $(-1, -1+1/(n-2)(n-3)^{n-3}]$.  Thus $\degpoly{CL_n}{x}$ has a root in the interval $(-(n-2), -(n-2)+1/(n-3)^{n-3}] = (-(n-2), -(n-2)+\epsilon_e(n)]$.

\end{proof}

Since we have shown there is at least one negative root when $n$ is even, there in fact must be two negative roots by what we found above with the Rule of Signs.  Using the IVT, we can quickly find that this root is in the interval $[-1, 0)$.  Evaluating the polynomial $g(x) = \degpoly{CL_n}{x}/x$ (just removing the known root at $0$) at these endpoints, we find

\begin{align*}
    g(0) & = (0)^{n-2} + (n-2)(0)^{n-3} + 1 \\
    & = 1 \\
    & > 0,
\end{align*}

\noindent and

\begin{align*}
    g(-1) & = (-1)^{n-2} + (n-2)(-1)^{n-3} + 1 \\
    & = (-1)^{n-3}(-1 + n - 2) + 1 \\
    & = -(n-3) + 1 \\
    & \leq 0,
\end{align*}

\noindent as $n \geq 4$.  Therefore $\degpoly{CL_n}{x}$ has a root in $[-1, 0)$ when $n$ is even.  In fact, this last inequality is an equality if and only if $n = 4$, when $\degpoly{CL_4}{x} = x^3 + 2x^2 + x = x(x+1)^2$.  In this case there is a double root at $-1$, which is why the half-closed interval is needed in Proposition \ref{prop:leaf_complete_real}.

A complex number $z$ is a {\em limit of zeros} of the sequence of polynomials $P_{1},P_{2},\ldots$ if there is a sequence $z_{1},z_{2},\ldots$ of complex numbers such that $P_{n}(z_{n}) = 0$ and $\displaystyle{\lim_{n \rightarrow \infty}  z_{n} = z}$. Suppose polynomials $P_{1},P_{2},\ldots$ satisfy a fixed term recurrence 
\begin{eqnarray} 
P_{n+k}(x) & = & -\sum_{i=1}^{k}f_{i}(x)P_{n+k-i}(x), \label{BKW}
\end{eqnarray}
where the $f_{i}$'s are polynomials in $x$.
We can solve such a recurrence to derive an explicit formula of the type
\begin{eqnarray} 
P_{n}(x) & = & \sum_{i=1}^{k} \alpha_i(x)(\lambda_{i}(x))^{n} \label{BKWexplicitformula}
\end{eqnarray}
where the $\lambda_i$'s are the zeros of the characteristic equation of the recursive relation \eqref{BKW}.
Beraha, Kahane and Weiss proved a beautiful result concerning the limits of the zeros of such polynomials, which we state as follows.

\begin{theorem}[Beraha-Kahane-Weiss, \cite{bkw1975}]

Suppose $\{P_t(x) : t \in \mathbb{N}\}$ is a sequence of polynomials having the form

\[ P_{t}(x) = \sum_{j=1}^s \alpha_j(x)\lambda_j(x)^t \]

\noindent for some polynomials $\alpha_j$ and not identically zero analytic functions $\lambda_j$, satisfying the following non-degeneracy condition: there is no constant $\omega$, with $|\omega| = 1$, such that $\lambda_i = \omega\lambda_j$ for some $i \neq j$.  Then $z$ is a limit of zeros for $\{P_t(x)\}$ if and only if the $\alpha_j$'s, $\lambda_j$'s can be reordered such that at least one of the following holds:

\begin{flalign*}
& 1. \;\; |\lambda_1(z)| > |\lambda_j(z)|, \;\; 2 \leq j \leq s, \;\; \mbox{and} \;\; \alpha_1(z) = 0  \\
& 2. \;\; |\lambda_1(z)| = |\lambda_2(z)| = \cdots = |\lambda_l(z)| > |\lambda_j(z)|, \;\; l+1 \leq j \leq s, \;\; \mbox{for some} \;\; l \geq 2.
\end{flalign*}

\label{thm:bkw}
\end{theorem}

 We can address the observation of roots converging to the unit circle, from Figure \ref{fig:leaf_complete_circ_roots}, with the extension of the BKW (Beraha-Kahane-Weiss) Theorem:

\begin{theorem}[\cite{brown2020extension}]

Let $\{ P_n(x) \}$ be a sequence of analytic functions of the form
\be
\label{eqn:P_n}
P_n(x)= \sum_{i=1}^k \alpha_i(n; x)(\lambda_i(x))^n,
\ee
where the $\lambda_i$ are analytic and not identically zero, $\lambda_i(x) \neq \omega\lambda_j(x)$ for any $\omega\in\mathbb{C}$ of unit modulus, and $\alpha_i(n; x)$ have the form
\be
\label{eqn:alphas}
\alpha_i(n; x) = n^{d_i} p_{i, d_i}(x) + n^{d_i - 1} p_{i, d_{i-1}}(x) + \cdots + n p_{i, 1}(x)+ p_{i, 0}(x).
\ee
where $d_i$ is the degree of $\alpha_i(n; x)$, the coefficient functions $p_{i, j}$ are analytic, and $p_{i, d_i}$ are not identically zero. 

Then $z\in\mathbb{C}$ is a limit of zeros of the family $\{P_n(x)\}$ if the $\lambda_i$ can be reordered such that either of the following conditions hold.
\begin{enumerate}
\item $|\lambda_1(z)| > |\lambda_i(z)|$ for all $i \neq 1$ and $p_{1, d_1}(z)=0$.
\item for some $l \geq 2$, $|\lambda_{1}(z)| = |\lambda_{2}(z)| = \cdots = |\lambda_{l}(z)| > |\lambda_{j}(z)|$ for all $j > l$ and there exists at least one $i$ such that $1 \leq i \leq l$ and $p_{i, d_i}(z) \neq 0$. 
\end{enumerate}

\label{thm:bkw_ext}
\end{theorem} 

Let us examine the limits of the roots of $\degpoly{CL_n}{x}$, as $n \to \infty$.  Since there is always a root at $x=0$, we can just consider the polynomial

\[    g_{n-3}(x) = \frac{\degpoly{CL_n}{x}}{x} = x^{n-3}(x + n - 2) + 1. \]

\noindent With a substitution of $N=n-3$, $g_N(x) = x^N(x+N+1)+1$ is in the form to apply Theorem \ref{thm:bkw_ext} if we let $\lambda_1(x) = x$, $\lambda_2(x) = 1$, $\alpha_1(N ; x) = x + N + 1$, and $\alpha_2(N ; x) = 1$.  Furthermore, we have $p_{1, 1}(x) = 1$ as the coefficient polynomial on $N$ in $\alpha_1(N ; x)$, and $p_{2, 0}(x) = 1$ as the coefficient polynomial on $N$ in $\alpha_2(N ; x)$.  Since both $p_{1, 1}(x)$ and $p_{2, 0}(x)$ are nonzero, we can rule out using the first condition of Theorem \ref{thm:bkw_ext} to find the limits.  The second condition immediately gives that the limits of $g_N(x)$ are the points $z$ where $|\lambda_1(z)| = |\lambda_2(z)|$, or where $|z| = 1$, i.e. the unit circle.  Thus we derive:

\begin{proposition}
The limits of the roots of $\degpoly{CL_n}{x}$, as $n \to \infty$, contain $0$ and the unit circle $|z| = 1$. \qed
\end{proposition}

We have now verified our observation that there are roots of $\degpoly{CL_n}{x}$ which approach the unit circle.  In fact, for $n \geq 5$, all the roots of $\degpoly{CL_n}{x}$ except for the real root located near $-(n-2)$  from Proposition \ref{prop:leaf_complete_real} (that is, within the interval $(-(n-2)-\epsilon_o(n), -(n-2))$ if $n$ is odd, or inside the interval $(-(n-2), -(n-2)+\epsilon_e(n))$ if $n$ is even) are contained within the unit circle.  This follows easily from Rouch\'{e}'s Theorem (see, for example, \cite{marden1949}), as $|x^{n-1}+x| \leq 2 < |(n-2)x^{n-2}| = n-2$ on $|z|=1$, and hence $\degpoly{CL_n}{x} = x^{n-1}+(n-2)x^{n-2}+x$ and $(n-2)x^{n-2}$ have the same number of roots inside the disk $|z|<1$, which is $n-2$.

\subsection{Anti-Regular Graphs}
\label{subsec:anti-reg}

At the other end of the spectrum, there are graphs of order $n$ having $n-1$ distinct degrees.  Indeed there cannot be more than $n-1$ distinct degrees: if there were $n$ distinct degrees, then each of $0, 1, \ldots, n-1$ would need to appear as the degree of exactly one vertex, and there would simultaneously be a vertex adjacent to all others (degree $n-1$) and a vertex not adjacent to any (degree $0$), a contradiction.
These graphs are called \emph{anti-regular} \cite{ali2020survey}, also known as \emph{quasi-perfect}, \emph{maximally non-regular}, \emph{degree anti-regular}, or \emph{half-complete} \cite{ali2020survey,fishburn1983packing,dimitrov2014total}. 

For a given $n \geq 2$, there are precisely two graphs (up to isomorphism) with $n-1$ distinct degrees (see \cite{dimitrov2014total}): first, the graph $H_n$, with degrees $1, 2,  \ldots, n-1$.  Every degree appears once in the degree sequence, except for $\lfloor n/2 \rfloor$, which appears twice.  $H_n$ can be formed by taking vertices $v_1,  \ldots, v_n$, and adding all edges of the form $\{ v_i, v_j \}$ such that $i+j \geq n+1$.  The other graph has degrees $0, 1,  \ldots, n-2$, and is the graph complement $H_n^c$ of $H_n$.  The degree which appears twice in the degree sequence in this case is $n-1-\lfloor n/2 \rfloor = \lfloor (n-1)/2 \rfloor$:


\noindent Thus we can easily write the degree polynomials for these graphs:

\[ \degpoly{H_n}{x}  = \sum_1^{n-1} x^i + x^{\lfloor n/2 \rfloor} = \frac{x(x^{n-1}-1)}{x-1} + x^{\lfloor n/2 \rfloor},\]

\noindent and

\[ \degpoly{H_n^c}{x} = x^{n-1}\degpoly{H_n}{1/x}  = \frac{x^{n-1}-1}{x-1} + x^{\lfloor (n-1)/2 \rfloor}.\]

\noindent These polynomials have no gaps in the powers of the terms with nonzero coefficients (that is, there are nonnegative integers $d$ and $k$ with $d \geq k$ such that the coefficient of $x^i$ is nonzero iff $d \geq i \geq k$), and have only a single term with coefficient greater than one.  See Figure \ref{fig:anti-regular} for some examples of anti-regular graphs and their degree polynomials.

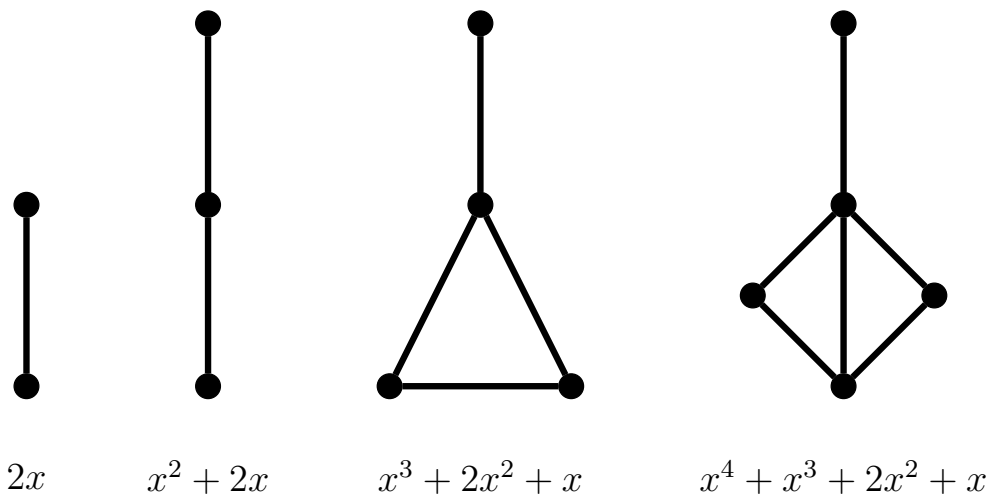
\begin{figure}[t!]
    \centering
    \resizebox{0.79\textwidth}{!}{\begin{tikzpicture}
        \centering
        \begin{pgfonlayer}{nodelayer}
            \node [style=node,scale=0.7] (21) at (0, 0) {};
            \node [style=node,scale=0.7] (22) at (0, 2) {};
            
            
            \node [style=node,scale=0.7] (31) at (2, 0) {};
            \node [style=node,scale=0.7] (32) at (2, 2) {};
            \node [style=node,scale=0.7] (33) at (2, 4) {};
            
            
            \node [style=node,scale=0.7] (41) at (4, 0) {};
            \node [style=node,scale=0.7] (42) at (6, 0) {};
            \node [style=node,scale=0.7] (43) at (5, 2) {};
            \node [style=node,scale=0.7] (44) at (5, 4) {};
            
            
            \node [style=node,scale=0.7] (51) at (8, 1) {};
            \node [style=node,scale=0.7] (52) at (9, 0) {};
            \node [style=node,scale=0.7] (53) at (10, 1) {};
            \node [style=node,scale=0.7] (54) at (9, 2) {};
            \node [style=node,scale=0.7] (55) at (9, 4) {};
            
            
            \node [style=none] (a) at (0, -1) {$2x$};
            \node [style=none] (b) at (2, -1) {$x^2+2x$};
            \node [style=none] (c) at (5, -1) {$x^3+2x^2+x$};
            \node [style=none] (d) at (9, -1) {$x^4+x^3+2x^2+x$};
        \end{pgfonlayer}
        \begin{pgfonlayer}{edgelayer}
            \draw [style=edge] (21) to (22);
            
            \draw [style=edge] (31) to (32);
            \draw [style=edge] (32) to (33);
            
            
            \draw [style=edge] (41) to (42);
            \draw [style=edge] (42) to (43);
            \draw [style=edge] (43) to (41);
            \draw [style=edge] (43) to (44);
            
            
            \draw [style=edge] (51) to (52);
            \draw [style=edge] (52) to (53);
            \draw [style=edge] (53) to (54);
            \draw [style=edge] (54) to (51);
            \draw [style=edge] (52) to (54);
            \draw [style=edge] (54) to (55);
            
        \end{pgfonlayer}
    \end{tikzpicture}}
    \caption{Examples of anti-regular graphs and their degree polynomials.  Left to right: $H_2, H_3, H_4, H_5$.}
    \label{fig:anti-regular}
\end{figure}


Figure \ref{fig:Hn_roots} shows the degree roots of the connected anti-regular graphs up to order $n = 50$.  Furthermore, we identify roots for even $n$ with red and those for odd $n$ with blue.  Some immediate observations are: when $n$ is even the (nonzero) roots appear to be on the unit circle, and not so for odd $n$.  However, the roots for odd $n$ surround the unit circle and possibly converge to it.  No root seems to exceed a modulus of $2$, which occurs for a real root.  

\begin{figure}
     \centering
     \begin{subfigure}
         \centering
         \includegraphics[width=0.68\textwidth]{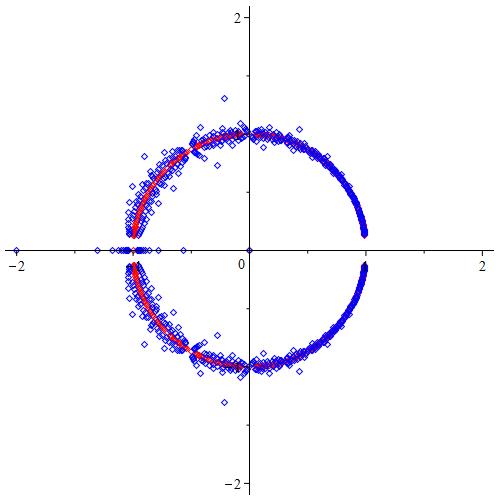}
         \label{fig:hn_rts_con}
     \end{subfigure}
        \caption{Degree roots for connected anti-regular graphs $H_n$, up to order $n=50$ (red roots correspond to even $n$, blue to odd $n$.}
        \label{fig:Hn_roots}
\end{figure}

As the degree roots of the disconnected anti-regular graphs are merely the inverses of the degree roots of the of their complements, we can restrict our attention to  the connected graphs $H_n$.  For a graph of order $n \geq 2$, the connected anti-regular graph $H_n$ has degree polynomial $\degpoly{H_n}{x} = \sum_{j=1}^{n-1} x^j + x^{\lfloor n/2 \rfloor}$.  As $\degpoly{H_n}{x} = x\frac{1-x^{n-1}}{1-x} + x^{\lfloor n/2 \rfloor},$ for all $x \neq 1$, the roots of $\degpoly{H_n}{x}$ are the solutions to the equation

\begin{equation} x - x^n + x^{\lfloor n/2 \rfloor} - x^{\lfloor n/2 \rfloor + 1} = 0 \label{eq:anti_reg_con} \end{equation}

\noindent except $x=1$.  We now examine the solutions to (\ref{eq:anti_reg_con}) via two cases on $n$.

\noindent \textbf{Case 1:} $n = 2k$, $k \geq 1$.  Here, (\ref{eq:anti_reg_con}) simplifies to

\[ x^{2k} + x^{k+1} - x^k - x = 0 \]

\noindent or

\[ x(x^{k}-1)(1+x^{k-1}) = 0. \]

\noindent Thus the roots of $\degpoly{H_{2k}}{x}$ are $x = 0$, the $k^{th}$ roots of unity (except for $1$ itself), and the $(k-1)^{th}$ roots of $-1$.

\noindent \textbf{Case 2:} $n = 2k+1$, $k \geq 1$.  In this case, (\ref{eq:anti_reg_con}) becomes

\begin{equation} x^{2k+1} + x^{k+1} - x^k - x = 0. \label{eq:anti_reg_con_case2} \end{equation}

\noindent These polynomials require numerical techniques to find their solutions.  However, we can deduce some information about them.  First of all, as $|x^{2k+1}| > |x^{k+1}| + |x^{k}| +x \geq |x^{k+1} - x^{k}  - x|$ on the circle $|x| \leq 2$, by Rouch\'{e}'s Theorem, the roots of $x^{2k+1} + x^{k+1} - x^k - x$ are in the disk $|x| = 2$.
Of course, there is a root at $0$.  Since there is exactly one sign change in the coefficients, there is exactly one positive solution to (\ref{eq:anti_reg_con_case2}) by the Rule of Signs.  Trivially, this solution is $x=1$, which is exactly the point we are excluding.  If we substitute $x \to -x$, we obtain

\[ - x^{2k+1}+(-1)^{k+1} x^{k+1}+(-1)^{k+1} x^k + x = 0. \]

\noindent Regardless if $k$ is even or odd, this equation has exactly one sign change and thus has exactly one positive root.  Therefore, $\degpoly{H_{2k+1}}{x}$, has exactly one negative root.  It follows that this root is in the interval $[-2,0)$. We can bound this negative root to the interval $[-2, -1/2)$ using the well-known Enestr\"{o}m-Kakeya Theorem:

\begin{theorem}[Enestr\"{o}m-Kakeya \cite{enestrom1920remarque,kakeya1912limits}]

Suppose $p(x) = a_n x^n + a_{n-1}x^{n-1} + \cdots + a_1 x + a_0$ where each $a_i$ is positive.  Let $q_k = a_{k-1}/a_k$, for $1 \leq k \leq n$.  Then any root $z$ of $p(x)$ satisfies

\[ \min_k \{q_k\} \leq |z| \leq \max_k \{q_k\}. \]  

\label{thm:kakeya}
\end{theorem}

As $\degpoly{H_n}{x} = x\sum_{j=0}^{n-2}x^j + x^{\lfloor n/2 \rfloor}$, the minimum ratio of consecutive coefficients is $1/2$.  By Theorem \ref{thm:kakeya}, every root of $\degpoly{H_n}{x}/x$ (i.e. the non-zero roots of $\degpoly{H_n}{x}$) have modulus in the interval $[1/2, 2]$.  This also holds for even $n$.  For the real root in the case of odd $n$ (which we were originally interested in), we can slightly improve this interval.  Evaluating the left hand side of (\ref{eq:anti_reg_con_case2}) at $x=-1/2$, we obtain

\begin{align*}
    x^{2k+1}+x^{k+1}-x^k-x & = \left( \frac{-1}{2} \right)^{2k+1} + \left( \frac{-1}{2} \right)^{k+1} - \left( \frac{-1}{2} \right)^{k} + \frac{1}{2} \\
    & = \frac{-1 + (-1)^{k+1}(2^k + 2^{k+1}) + 2^{2k}}{2^{2k+1}}.
\end{align*}

\noindent If $k$ is odd, then this quantity is clearly positive.  When $k$ is even, observe

\[  -1 + (-1)^{k+1}(2^k + 2^{k+1}) + 2^{2k} = -1 - (2^k + 2^{k+1}) + 2^{2k} > 0.\]

Thus in any case, the left hand side of (\ref{eq:anti_reg_con_case2}) is positive when $x = -1/2$.  Therefore the negative root of equation (\ref{eq:anti_reg_con_case2}) is actually in the interval $[-2, -1/2)$.










We can also study the limits of degree roots for $H_n$, as $n \to \infty$, using the BKW Theorem (Theorem~\ref{thm:bkw}) described earlier.
The first step  is writing our polynomials in the correct form.  Note that

\[ (1-x)\degpoly{H_n}{x} = x(1-x^{n-1}) + (1-x)x^{\lfloor n/2 \rfloor},\]

\noindent so we examine the limits of the roots of the latter. As before, we shall consider cases on the parity of $n$.\\

\noindent \textbf{Case 1:} $n = 2k$, $k \geq 1$.  Define the polynomial

\[  f_k(z) = (1-z)\degpoly{H_{2k}}{z} = - z^{2k} + (1-z) z^k + z.\]

\noindent Observe that $f_k(z)$ is readily in the form to apply the (original) BKW Theorem by setting $\alpha_1(z) = -1$, $\alpha_2(z) = 1-z$, $\alpha_3(z) = z$, and $\lambda_1(z) = z^2$, $\lambda_2(z) = z$, and $\lambda_3(z) = 1$.  We find that  the limits of the roots for $f_k(z)$, and also $\degpoly{H_{2k}}{z}$ are the unit circle centered at $z = 0$ and the origin.\\

\noindent \textbf{Case 2:} $n = 2k+1$, $k \geq 1$.  Similar to the first case, we can apply the BKW Theorem to the function

\[ g_k(z) = (1-z)\degpoly{H_{2k+1}}{z} = - z^{2k+1} + (1-z) z^k + z\]

\noindent by setting $\alpha_1(z) = -z$, $\alpha_2(z) = 1-z$, $\alpha_3(z) = z$, and $\lambda_1(z) = z^2$, $\lambda_2(z) = z$, $\lambda_3(z) = 1$.  The only difference between $g_k(z)$ and $f_k(z)$ is $\alpha_1$; thus the limits of the roots will be the same except for possibly those from the constraint $|\lambda_1(z)| > |\lambda_2(z)|$, $|\lambda_1(z)| > |\lambda_3(z)|$, and $\alpha_1(z) = 0$.  A quick verification determines that the limits of the roots of $g_k(z)$ are the same as $f_k(z)$, which overall gives that the limits of the roots of $\degpoly{H_n}{z}$, the unit circle centered at $z = 0$ and the origin. Summarizing:

\begin{proposition}
The degree polynomial of the anti-regular graph $H_{n}$ has exactly one negative root which lies in the interval $[-2,-1/2)$. The limits of roots of the degree polynomials $D(H_{n};x)$ are $0$ and the unit circle centered at $z = 0$. \qed
\end{proposition}



\section{Bounds on Degree Roots}
\label{sec:bounds}


We mentioned earlier that the set of all degree roots of graphs is equal to the set of all degree roots of multigraphs, which in turn is equal to the set of roots of non-negative integer coefficient polynomials. However, what if we restrict to graphs and multigraphs of order $n$? The question becomes much more interesting. In such a case, our point of comparison is with the roots of non-negative integer coefficient polynomials whose coefficient sum is $n$, as the sum of the coefficients in any degree polynomial of a multigraph of order $n$ is $n$. 

Even for some small $n$, there is a divergence, as the only degree root of a multigraph (or graph) of order $2$ is $0$, while $-1$ is a root of $1+x$. For order $3$, the only degree roots of (simple) graphs are $0$, $-1/2$ and $-2$, while the degree roots of multigraphs of order $3$ contain all roots of $1+2x^{\mu}$ for $\mu \geq 1$, and hence in their closure contain the unit circle centered at the origin. In fact, the same argument can be used to show that for all $n \geq 3$, there are infinitely many degree roots of multigraphs that are not degree roots of (simple) graphs. Take any graph $G$ of order $n \geq 3$ with an edge $e$. Then it is easy to see that if $G_{\mu}$ is the multigraph (of order $n$) formed from $G$ by replacing $e$ by a bundle of $\mu \geq 1$ parallel edges, then the degree polynomial of $G_{\mu}$ has the form of $a(x) \cdot x^{\mu} + b(x)$, where $a(x)$ and $b(x)$ are fixed nonzero polynomials. The BKW Theorem now shows that the limits of these degree roots contain the unit circle centered at the origin, while of course there are are only finitely many degree roots of graphs of order $n$ (as there are only finitely many such graphs). We shall see that for all even $n$, there are roots of non-negative integer coefficient polynomials whose coefficients sum to $n$ that are not the degree roots of multigraphs, so there is divergence there as well.




We will now consider bounding degree roots in terms of multigraph order, $n$.  
%
%
%
%
%
A well known bound on the zeros of polynomials \cite{Lagrange}, states that a real polynomial $\displaystyle{f(x) = \sum_{i=0}^{n} a_{i} x^i}$ has all its roots in the disk centred at the origin of radius $\displaystyle{R = \mbox{max} \left\{ \frac{|a_{n-1}| + \cdots + |a_{0}|}{a_{n}},1 \right\}}.$ The following bound immediately follows.

\begin{proposition}

Let $p(x) = a_{\Delta}x^{\Delta} + \cdots + a_{\delta}x^{\delta} \in \mathbb{Z}_{\geq 0}[x]$ where $\Delta > \delta$ and with $a_{\Delta}, a_{\delta} \geq 1$.  Suppose that $p(1) = a_{\Delta}+ \cdots + a_{\delta} = n$.  If $z$ is a root of $p(x)$, then

\[ |z| \leq \max\left\{ \frac{n-a_{\Delta}}{a_{\Delta}}, \frac{a_{\Delta}}{n-a_{\Delta}} \right\}. \]

\label{prop:circ_bound}
\end{proposition}

In Figure \ref{fig:deg_roots_2to10} we observe that the degree roots for graphs of order $n$ seemed to never exceed a modulus of $n-1$, and roots that had such a modulus were real.  

\begin{theorem}

If $z$ is nonzero degree root of a multigraph of order $n$, then 
\[ \frac{1}{n-1} \leq |z| \leq n-1.\] \hfill\qed

\label{cor:ord_n_bound}
\end{theorem}

This modulus upper bound is in fact true, even for multigraphs of order $n$, and follows directly from Proposition~\ref{prop:circ_bound}, when $a_{\Delta} = 1$ or $n-1$ (the lower bound holds as the set of nonzero degree roots of multigraphs is closed under inverses).

Figure \ref{fig:deg_roots_2to10} seems to suggest that the only roots of modulus $n-1$ are real, and only appear when $n$ is odd.  The following propositions address these observations for $n \geq 4$, since all degree roots for $n=2$ or $n=3$ are already real.

\begin{proposition}

Let $G$ be a graph of order $n \geq 4$, and suppose $\degpoly{G}{x}$ has a degree root $z$, where $|z| = n-1$.  Then $\degpoly{G}{x}$ has the form $\degpoly{G}{x} = x^{\Delta}+(n-1)x^{\Delta-1}$, and in particular $z = -(n-1)$.

\label{prop:n-1_mod_root}
\end{proposition}

\begin{proof}
Let $\degpoly{G}{x} = a_{\Delta}x^{\Delta} + \cdots + a_{\delta}x^{\delta}$ ($a_{\Delta} > 0$) be the degree polynomial of $G$ which has the root $z$ of modulus $n-1$ (so $G$ is not regular and hence $\delta < \Delta$).  Consider the theorem of Cauchy (see, for example, \cite{marden1949}) which states that all roots of $\degpoly{G}{x}$ have modulus strictly less than

\[ 1 + \max_{k \neq \Delta} \left\{ \left|\frac{a_k}{a_{\Delta}}\right| \right\} = 1 + \frac{\max_{k \neq \Delta} \{ a_k \}}{a_{\Delta}}. \]

\noindent Since this applies to the root $z$ with modulus $n-1$, we must have

\[ n-2 < \frac{\max_{k \neq \Delta} \{ a_k \}}{a_{\Delta}}. \]

\noindent As the non-negative coefficients sum to $n$, this inequality is only satisfied when $a_{\Delta}=1$ and $\max_{k \neq \Delta} \{ a_k \} = n-1$, so $a_{\Delta}=1$, $a_k = n-1$ for some $k < \Delta$, and all other coefficients are zero.  This gives $\degpoly{G}{x}$ the form $\degpoly{G}{x} = x^{\Delta}+(n-1)x^k$. As the modulus of the root $z$ is $n-1$, it follows that $\Delta-k=1$, and thus $k = \Delta-1$.  Thus $\degpoly{G}{x} = x^{\Delta}+(n-1)x^{\Delta-1}$, and we also conclude that $z=-(n-1)$.
\end{proof}

Polynomials of the form $x^{\Delta}+(n-1)x^{\Delta-1}$ are not necessarily degree polynomials for all values of $n$ and $\Delta$.  The next proposition tells us precisely when they are.

\begin{proposition}

A polynomial of the form $x^{\Delta}+(n-1)x^{\Delta-1}$ with $n \geq 4$, $\Delta \leq n-1$, is a degree polynomial of a (multi)graph of order $n$ if and only if $n$ is odd and $\Delta$ is even.

\label{prop:2term_when_is_degpoly}
\end{proposition}

\begin{proof}
$(\implies)$ We first prove the forward direction.  Since $x^{\Delta}+(n-1)x^{\Delta-1}$ is the degree polynomial of a multigraph, we know that the sum of the degrees must be even.  Hence, $\Delta+(n-1)(\Delta-1)$ is even, implying $\Delta$ and $(n-1)(\Delta-1)$ are of the same parity.  It is not hard to see that this implies $n$ is odd and $\Delta$ is even.

\noindent $(\impliedby)$ Suppose $n$ is odd, so $n=2k+1$ for some $k \geq 2$, and also that $\Delta$ is even, so $\Delta = 2d$ for some $1 \leq d \leq k$. We construct a graph $F_{n,\Delta}$ of order $n$ as follows. Take vertices $v_{0},v_{1},\ldots,v_{n-1}$. Arrange the vertices $v_{0},v_{1},\ldots,v_{n-2}$ cyclically and for $i \in \{0,1,,\ldots,n-2\}$, join each $v_{i}$ to $\Delta/2 - 1 = d-1$ on either side of it in the arrangement and as well as to $v_{i+k}$ (arithmetic modulo $n-1$). The subgraph induced on $v_{0},v_{1},\ldots,v_{n-2}$ is clearly $(\Delta-1)$-regular. Now take any matching $M$ of cardinality $d = \Delta/2$ in this subgraph, join $v_{n-1}$ to the ends of the edges in $M$, and then remove the edges in $M$. Call the resulting simple graph $F_{n,\Delta}$. Then $F_{n,\Delta}$ has one vertex of degree $\Delta$ and the remaining $n-1$ vertices of degree $\Delta-1$, and hence has degree polynomial $x^{\Delta}+(n-1)x^{\Delta-1}$.

\end{proof}

Proposition \ref{prop:2term_when_is_degpoly} also confirms what we stated earlier, namely that for even $n \geq 4$, there are roots of non-negative integer coefficient polynomials with sum $n$ that are not the degree roots of multigraphs of order $n$ (in particular, $-(n-1)$, which is a root of $x+(n-1)$).  Some examples of graphs that have a degree root at $-(n-1)$, when $n$ is odd, can be constructed by removing a perfect matching from $K_{2k}$ (where $k = (n-1)/2$), and adding a universal vertex. Another set of examples (which are disconnected) is afforded by taking the disjoint union of $P_3$ with $(n-3)/2$ copies of $P_2$.

We can also provide a bound for purely imaginary degree roots.

\begin{proposition}

If $z$ is a nonzero purely imaginary degree root, then

\[ \frac{1}{\sqrt{n-1}} \leq |z| \leq \sqrt{n-1}. \]


\label{prop:imag_root_bound}
\end{proposition}

\begin{proof}
Since $z$ lies on the imaginary axis, we have $z=ir$ for some $r \in \mathbb{R}$.  Let us write $\degpoly{G}{x} = a_{\Delta}x^{\Delta}+ \cdots + a_{\delta}x^{\delta}$.  $\degpoly{G}{ir} = 0$ can be written as

\[ i^{\Delta}\left( a_{\Delta}r^{\Delta}-a_{\Delta-2}r^{\Delta-2}+\cdots \right)+i^{\Delta-1}\left( a_{\Delta-1}r^{\Delta-1}-a_{\Delta-3}r^{\Delta-3}+\cdots \right) = 0, \]

\noindent or simply $i^{\Delta}A+i^{\Delta-1}B=0$.  Therefore, both $A$ and $B$ must be equal to zero.  Since $a_{\Delta} \geq 1$, there must be another coefficient ($a_k$, for some $k$) in $A$ that is nonzero.  Let us now consider two cases on the parity of $\Delta$.

\noindent \textbf{Case 1:} $\Delta=2k$.  In this case, we may write $A=0$ as

\[ a_{2k}r^{2k}-a_{2k-2}r^{2k-2}+\cdots=0. \]

\noindent Setting $s=r^2$ we have

\[ a_{2k}s^k-a_{2k-2}s^{k-1}+\cdots = 0, \]

\noindent and thus $-s$ is a root of $f(x) = a_{2k}x^k+a_{2k-2}x^{k-1}+\cdots$.  Since $f(x)$ has only non-negative integer coefficients, we apply Proposition \ref{prop:circ_bound}: $f(1) \leq n$ and $1 \leq a_{2k} \leq n-1$, so

\begin{align*}
    |s| & = |-s| \\
    & \leq \max \left\{ \frac{n-a_{2k}}{a_{2k}}, \frac{a_{2k}}{n-a_{2k}} \right\} \\
    & \leq n-1.
\end{align*}

\noindent Therefore, $|z| = |r| \leq \sqrt{n-1}$.

\noindent \textbf{Case 2:} $\Delta = 2k+1$.  In this case, we may write $A=0$ as

\[ a_{2k+1}r^{2k+1}-a_{2k-1}r^{2k-1}+\cdots = 0. \]

\noindent Dividing by $r$ and again setting $s=r^2$ we have

\[ a_{2k+1}s^k - a_{2k-1}s^{k-1}+\cdots = 0, \]

\noindent so $-s$ is a root of $g(x) = a_{2k+1}x^{k} + a_{2k-1}x^{k-1}+\cdots$.  As above, we can apply Proposition \ref{prop:circ_bound} to obtain $|z| \leq \sqrt{n-1}$.

Therefore, in any case, we have $|z| \leq \sqrt{n-1}$.  The lower bound follows since $1/z$ is a degree root as well.  
\end{proof}

\section{Open Problems}
\label{sec:concluding}

We conclude this paper by discussing a modulus bound conjecture that generalizes both Theorem \ref{cor:ord_n_bound} and Proposition \ref{prop:imag_root_bound}, and gives a tighter modulus bound in all other places.  We have seen that the disk $|z|  = n-1$ contains the degree roots of all multigraphs of order $n$, with a root on the boundary when $n$ is odd. However, Figure~\ref{fig:deg_roots_2to10} suggests, at least for graphs, that a smaller region might suffice. We propose the following.

\begin{conjecture}
    If $z$ is a nonzero degree root of a graph of order $n$ with argument $\arg (z) \in (-\pi, \pi]$, then

    \[ |z| \leq (n-1)^{\left|\frac{\arg (z)}{\pi}\right|}. \]
\end{conjecture}

This bound agrees with Theorem \ref{cor:ord_n_bound} for real roots ($\arg (z) = \pi$) and Proposition \ref{prop:imag_root_bound} for imaginary roots ($\arg (z) = \pm\pi/2$).  Figure \ref{fig:mod_bound_conjecture} shows the curve $|z| = (n-1)^{\left|\frac{\arg (z)}{\pi}\right|}$ along with the circular bound $|z| = n-1$ on plots of degree roots for some small values of $n$.

As evidence in favour of this conjecture, we have verified it for graphs of order $n \leq 9$ and trees of order $n \leq 18$.  Trivially, the degree roots for regular graphs are within this bound.  Similarly, the degree roots for graphs with only two degrees are within this bound.  Since the curve $|z| = (n-1)^{\left|\frac{\arg (z)}{\pi}\right|}$, for $n \geq 3$, lies exterior to the unit circle at all points except at $z=1$ (where it meets the unit circle), our conjecture also holds for degree roots inside the unit circle.  Thus our conjecture holds for the degree roots of the graphs $CL_n$.  

We also observe that the star $K_{1,n-1}$ has degree polynomial $x^{n-1} + (n-1)x$, with a degree root at $0$ and the rest having modulus $(n-1)^{1/(n-2)}$. One of these roots is $z = (n-1)^{1/(n-2)}e^{i\pi/(n-2)}$ with modulus $(n-1)^{\left|\frac{\arg (z)}{\pi}\right|}$, landing up right on the boundary of the curve.

\begin{figure}[ht!]
    \centering
    \begin{subfigure}
         \centering
         \includegraphics[width=0.38\textwidth]{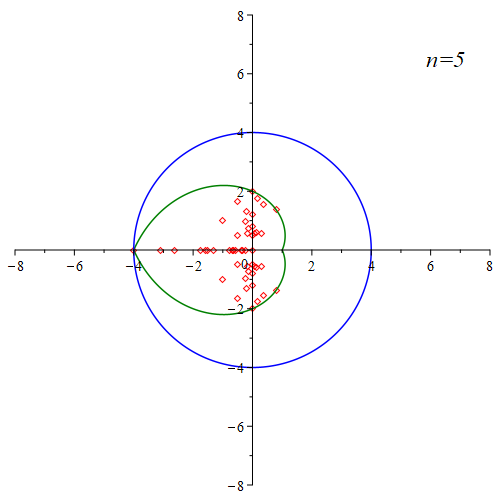}
     \end{subfigure}
     \begin{subfigure}
         \centering
         \includegraphics[width=0.38\textwidth]{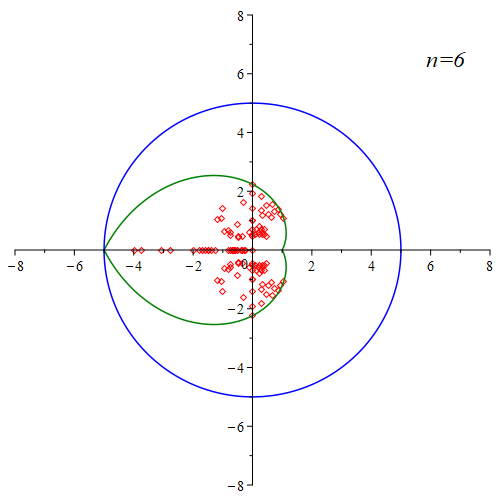}
     \end{subfigure}
     \begin{subfigure}
         \centering
         \includegraphics[width=0.38\textwidth]{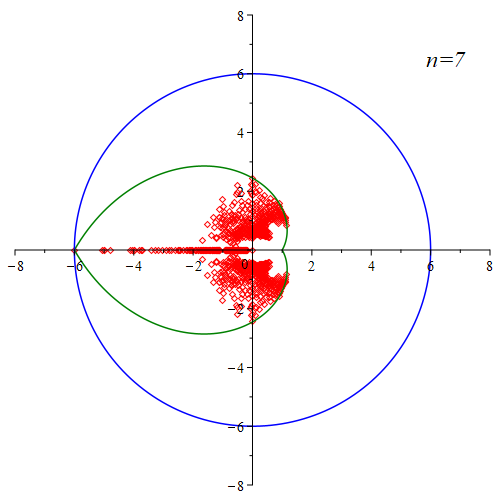}
     \end{subfigure}
    \caption{Degree roots of graphs for some small values of $n$ (red).  The blue curves show $|z| = n-1$, while the green curves are $|z| = (n-1)^{\left|\frac{\arg (z)}{\pi}\right|}$.}
    \label{fig:mod_bound_conjecture}
\end{figure}

Finally, we can investigate the degree roots of other families of graphs, both in terms of the conjecture, as well as in their own right. An examination of the degree roots of trees and certain multipartite graphs has been undertaken in \cite{georgethesis}. Future work on degree roots includes as well determining the set of rational degree roots of graphs and multigraphs, as well as bounds for the real and imaginary parts of degree roots of multigraphs of order $n$.




\section*{Acknowledgements}
J.I. Brown acknowledges research support from the Natural Sciences and Engineering Research Council of Canada (NSERC), grant RGPIN 2018-05227. 


\bibliographystyle{plain}

\end{document}